\documentclass[a4paper, 12pt]{article}

\usepackage[utf8]{inputenc}
\usepackage{amssymb,amsmath}
\usepackage{amsthm}
\usepackage[a4paper, top=2.5cm,bottom=2.5cm,left=2.5cm,right=2.5cm, marginparwidth=1.75cm]{geometry}
\usepackage[normalem]{ulem}
\usepackage{mathtext}
\usepackage{amsfonts}
\usepackage{epstopdf}
\usepackage{csquotes}
\usepackage{xcolor}
\usepackage{hyperref}
\usepackage{tikz}
\usetikzlibrary{calc}
\usetikzlibrary{external}
\usepackage{pgfplots}
\pgfplotsset{compat=1.15}
\usepackage{mathrsfs}
\usetikzlibrary{arrows}

\newtheorem{definition}{Definition}
\newtheorem{theorem}{Theorem}
\newtheorem{lemma}{Lemma}
\newtheorem{prop}{Proposition}
\newtheorem{corollary}{Corollary}
\newtheorem{conj}{Conjecture}
\newtheorem*{theorem*}{Theorem}
\newtheorem{remark}{Remark}

\newcommand{\bb}{\mathbb}
\newcommand{\mbf}{\mathbf}
\renewcommand{\to}{\rightarrow}

\renewcommand{\cal}{\mathcal}
\renewcommand{\leq}{\leqslant}
\renewcommand{\geq}{\geqslant}

\begin{document}
\title {Abelian Nivat's conjecture for non-rectangular patterns}

\author{Nikolai Geravker$^1$ \textit{kolya-ger@yandex.ru} 
\and Svetlana Puzynina$^{1,2}$ \textit{s.puzynina@gmail.com}
}

\date{%
    $^1$Saint Petersburg State University, Russia\\%
    $^2$Sobolev Institute of Mathematics, Russia\\[2ex]%
    \today
}

\maketitle

\definecolor{xdxdff}{rgb}{0.49019607843137253,0.49019607843137253,1}
\definecolor{ududff}{rgb}{0.30196078431372547,0.30196078431372547,1}
\definecolor{cqcqcq}{rgb}{0.7529411764705882,0.7529411764705882,0.7529411764705882}

	
	\begin{abstract}In this paper, we study the relation between periodicity of two-dimensional words and their abelian pattern complexity. A pattern $\cal{P}$ in $\mathbb{Z}^n$ is the set of all translations of some finite subset $F$ of $\mathbb{Z}^n$. An $F$-factor of an infinite word is a finite word restricted to $F$. Then the pattern complexity over a pattern $\cal{P}$ counts the number of distinct $F$-factors of an infinite word, for $P\in \cal{P}$. Two finite words are called abelian equivalent if for each letter of the alphabet, they contain the same numbers of occurrences of this letter. The abelian pattern complexity counts the number of $F$-factors up to abelian equivalence. As the main result of the paper, we characterize two-dimensional convex patterns with the following property: if abelian pattern complexity over a pattern $\cal{P}$ is equal to 1, then the word is fully periodic. Similar result holds for a function on $\mathbb{Z}^2$ instead of a word and for constant sums instead of abelian complexity equal to 1.
	In dimensional 1, we characterize patterns for which there exist non-constant functions with constant sums. 
	\end{abstract}
	
	\section{Introduction}
	
	One of the most studied topics in combinatorics on words is the  of infinite words, both in the one-dimensional and in the multidimensional case. An infinite one-dimensional word $w$ is an infinite sequence of symbols from a finite set called the alphabet. Its complexity is defined as a function $p(n)$ which counts, for each integer $n$, the number of its distinct factors (i.e., blocks of consecutive letters) of length $n$. The complexity of an infinite word provides a useful measure of the extent of randomness of the word and more generally of the subshift it generates. For example, periodic words have bounded factor complexity, while digit expansions of normal numbers have maximal complexity. The notion of a complexity of infinite words is closely related to the notion of topological entropy of discrete dynamical systems \cite{AKA}.  In fact, for a one-dimensional subshift, and in particular for a subshift generated by an infinite word, its  topological entropy is given by the exponential growth
rate of its complexity.
	
	First steps in complexity theory of words were made by Morse and Hedlund in 1938. They proved that if for an infinite word there exists $n$ such that $p(n) \leq n$,  
then the word is periodic  \cite{MoHe1}.
This result is classical in combinatorics on words and is referred to as Morse and Hedlund theorem. This theorem implies that the minimal complexity of an aperiodic word is at least $n+1$. Words of complexity $n+1$ for each $n$ exist and they are called Sturmian words (see, e.g., \cite[Chapter~2]{Lo}). The family of Sturmian words has been widely studied for their
theoretical importance and applications to various fields of science. They admit several equivalent
characterisations of algebraic, arithmetic and 
geometrical nature. For example, in \cite{MoHe2} Hedlund and Morse showed that each Sturmian word may be
realized geometrically by an irrational rotation on the circle.

	Nivat's conjecture  \cite{Nivat} is a generalization of Morse and Hedlund theorem 
	to two dimensions. A two-dimensional word (or a configuration) $\mbf{w}$ is an element of $A^{\bb{Z}^2}$, where $A$ is a finite set called an alphabet. A word $\mbf{w}$ is called \emph{
	periodic} if there exist a vector $p \in \bb{Z}^2$ such that $\mbf{w}(x) = \mbf{w}(x + p)$ for all $x \in \bb{Z}^2$. A complexity of a two-dimensional word $\mbf{w}$ is a function
$p_{\mbf{w}}(m, n)$ counting for each $m, n \in \bb{N}$ the number of distinct rectangular $m \times n$ blocks.
	
	\begin{conj}[Nivat, 1997]\label{Nivat'conj}
		Let $\mbf{w}$ be a two-dimensional word. If there exists $m, n$ such that $p_{\mbf{w}}(m, n) \leq mn$, then $\mbf{w}$ is periodic. 
	\end{conj}


	Two-dimensional words of complexity $mn+1$ for each $m$, $n$ exist and they have been characterized in \cite{JC99}. Remarkably, their structure completely different from Sturmian words.	

Nivat's conjecture remains open despite of efforts of different scientists. It has been proven is some weak forms, for example, in an asymptotic form  by  Kari, Szabados \cite{KS2015}. 
	In \cite{ECM03} it was shown that $P_{\mbf{w}}(m, n)\leq mn/144$ is enough to guarantee the periodicity of $\mbf{w}$. 
	This bound has been improved to $P_{\mbf{w}}(m, n) \leq mn/16$ in \cite{QZ04}, and recently to $P_{\mbf{w}}(m, n) \leq mn/2$ in \cite{CK2015} using dynamical systems approach. The problem can be translated to a dynamical one as follows.

 We endow the alphabet $A$ with the discrete topology, $A^{\mathbb{Z}^d}$ with the product topology, and define a $\mathbb{Z}^d$-action on $X = A^{\mathbb{Z}^d}$ by $(T^u \mbf{w})(x) = \mbf{w}(x + u)$ for $u \in \mathbb{Z}^d$. With respect to this topology, the maps $T^u : X \to X$ are continuous. Let $O(\mbf{w}) = \{T^u \mbf{w} | u \in \mathbb{Z}^d\}$ denote the $\mathbb{Z}^d$-orbit of $\mbf{w}\in X$, and $\overline{O} (\mbf{w})$ its closure. In this dynamical setting, one can rephrase periodicity. The statement that $\mbf{w}$ has a periodicity vector is equivalent to saying that $\mathbb{Z}^d$ does not act faithfully on $\overline{O} (\mbf{w})$. A word $\mbf{w}$ is fully periodic if it has $d$ linearly independent periodicity vectors, which means that  $\overline{O} (\mbf{w})$ is finite. The dynamical systems approach from \cite{CK2015} makes use of the nonexpansive subspaces of this action. A similar method has been used to prove the conjecture for small rectangle sizes \cite{CK2016}. 
For some other versions of Nivat's conjecture and minimal complexity in two or more dimensions we refer to \cite{Ca99,DR11,KM18}.

	The main objective of this paper
	is finding generalizations of Morse and Hedlund theorem 
	and Nivat's Conjecture 
	for abelian pattern complexity. 
	We recall that two finite words $u$ and $v$ are said to be \emph{abelian equivalent} if $|u|_a = |v|_a$ for all $a \in A$, where $|u|_a$ denotes the number of occurrences of the letter $a$ in $u$. 
 The \emph{abelian complexity} $a_\mbf{w}(n)$ of a  (one-dimensional) word $\mbf{w}$ is the function counting number of distinct abelian classes of factors of length $n$. 
	This definition can be extended to two or more dimensions in a natural way as a function counting the number of abelian classes of rectangular blocks.  
	
	For one-dimensional words an abelian analogue of Morse and Hedlund theorem 
	is straightforward: clearly, the condition $a_{\mbf{w}}(n) = 1$ implies that $\mbf{w}$ is $n$-periodic. Aperiodic words of abelian complexity $2$ for each $n$ exist and surprisingly the set of words with this property coincides with the family of Sturmian words \cite{CoHe73}. So, among aperiodic words, Sturmian words have minimal complexity both in classical and in abelian sense.  The study of abelian complexity of one-dimensional infinite words has been developed, e.g., in \cite{KWZ15,MR13,RCZ11,S09}. For two-dimensional words, the abelian modifications of Nivat's conjecture has been studied  in  \cite{P19}. It has been shown that there exists an aperiodic word $\mbf{w}$ and integers $m$ and $n$ such that $a_{\mbf{w}}(m, n) = 1$. However, for an aperiodic recurrent two-dimensional word $\mbf{w}$ there exist infinitely many pairs of numbers $m, n$ such that $a_{\mbf{w}}(m, n) \geq 3$. 

    	The notions of a complexity and an abelian complexity can be extended to any pattern. More precisely, we call any finite subset $F$ of $\mathbb{Z}^n$ a {\emph{figure}}, and a pattern $\cal{P}$ is a set of all integer translations of some figure $F$. We can fix some order on the elements of $F$: $F=\{ x^1, \ldots, x^k\}$. Then for an infinite $n$-dimensional  word ${\bf w}$ the word ${\bf w}_{x^1}\cdots {\bf w}_{x^k}$ is called an $F$-factor of ${\bf w}$. The \emph{pattern complexity} over the pattern $\cal{P}$ then counts the number of $F$-factors of the word ${\bf w}$, for all $F\in \cal{P}$. Similarly, the \emph{abelian pattern complexity} $a_{\mbf{w}}(\cal{P})$ over the pattern $\cal{P}$ counts the number of abelian classes of $F$-factors of the word ${\bf w}$, for all $F\in \cal{P}$.

		A related concept of maximal pattern complexity has been introduced by Kamae and Zamboni in 2002 \cite{KZ2002}. The maximal pattern complexity $p^*_{\bf w}$ of a word ${\bf w}$ is defined as a function counting, for each $k$, the supremum of the pattern complexities for patterns defined by figures of size $k$. Similarly to factor complexity and abelian complexity, the maximal pattern complexity also gives a characterization of periodicity in the one-dimensional case: An infinite one-dimensional word ${\bf w}$ is eventually periodic if and only if $p^*_{\bf w}(k)<2k$ for some integer $k$ \cite{KZ2002}.  The abelian maximal pattern complexity also gives a characterization of aperiodicity in terms of so-called aperiodicity by projection \cite{KRX06,KWZ15}.

The main problem we study in the paper is the following: find a characterization of patterns such that for each word $\mbf{w}$ the condition $a_{\mbf{w}}(\cal{P}) = 1$ implies that $\mbf{w}$ is a periodic word. We call such patterns \emph{abelian rigid}. 
We remark that the set of $d$-dimensional words of abelian pattern complexity equal to $1$ for a pattern $\cal{P}$ forms a subshift, i.e. a closed $T$-invariant subset of $A^{\mathbb{Z}^d}$. Moreover, this subshift is a subshift of finite type, i.e., subshift defined by local constraints.
We also consider a variant of this problem when instead of a word we consider a function with finitely many values, and instead of abelian pattern complexity equal to $1$ we consider constant sums in the figures of the pattern. Clearly, such a function can also be considered as a word with letters corresponding to values of the function, and if we prove under certain conditions periodicity of all functions with finitely many values, the periodicity of words in these conditions follows immediately. 

 In terminology of integer-valued functions on $\mathbb{Z}^n$ with constant sums in certain patterns related problems have already been considered in the literature. One of them is a discrete analogue of Pompeiu's problem, which can be stated as follows. Let $\Omega$ be a finite set of patterns
of $\mathbb{Z}^n$.
The problem is to determine when the only function $\varphi:
\mathbb{Z}^n\to \mathbb{Z}$ such that
$$\sum_{ x \in F } \varphi({ x})=0 \mbox{ for every } F \in \cal{P}, \cal{P}\in \Omega  $$
is the zero function (see \cite{Zeil78}). The paper \cite{AP08} is focused on
the cases when $\Omega$ consists of one ball or one sphere.

It is also natural to consider patterns with weights in the following sense. 
Each point $x$ from a figure $F^w$ is equipped with a real number $g_x$ considered as its weight, and all translations of this figure define a pattern $\cal{P}^w$  with weights. We then consider functions for which $\sum\limits_{x \in F^w} \mbf{w}(x)g_x = 0$ for each $ F^w \in \cal{P}^w$.

In the one-dimensional case, it is easy to see that $\mbf{w}$ is periodic if $a_{\mbf{w}}(\cal{P}) = 1$. 
A natural question is, for which patterns there exists a non-unary word with abelian pattern complexity 1 over this pattern? In this paper, we obtain a necessary condition for that, and a characterization in terms of sums functions. 
 To formulate the condition, we associate a polynomial to a pattern $\cal{P}$ as follows. We choose some figure $F_{\cal{P}} \in\cal{P}$ (for example such that the minimal coordinates of the points of the figure are equal to 0). Then the \emph{polynomial of the pattern} $\cal{P}$ is defined by
	$$
	    Poly_{\cal{P}}(x) = \sum_{(t, g_t) \in F_{\cal{P}}} g_t x^t,
	$$
	where $x^t = x_1^{t_1} x_2^{t_2} \cdots x_n^{t_n}$.	
	
	\begin{theorem*}
		Let $\cal{P}$ be a pattern with weights. 
		Then there exists a non-zero function $\mbf{w}:\mathbb{Z}\to \mathbb{C}$ such that
		\[
			\sum\limits_{x \in F} \mbf{w}(x)g_x = 0
		\mbox{ for all } F \in \cal{P}, 
		\]
	 if and only if there exists $n$ such that $Poly_{\cal{P}}(x)$ is divisible by $n$'th cyclotomic polynomial $\Phi_n(x)$ in the ring $\bb{Q}[x]$.
	\end{theorem*} 
	
	 We say that a pattern is {\emph{abelian rigid}}  the abelian pattern complexity over this pattern is equal to 1 only for fully periodic words. 	As the main result of the paper, we obtain the following characterization of abelian rigid  two-dimensional convex patterns:
	
	
	\begin{theorem*}
		A convex two-dimensional pattern $\cal{P}$ is abelian rigid if and only if $Poly_{\cal{P}}(x,y)$ does not have a divisor of the form 
		$\sum\limits_{i = 0}^l x^{iv_1} y^{iv_2}$ for some integers $v_1$, $v_2$ and for some positive integer $l$.
	\end{theorem*}	
	
	

	\section{Definitions and notation}

In the paper, we mostly follow the usual terminology of combinatorics on words from \cite{Lo}.  Let $A$ be a finite set and $n$ be an integer. A function from $\bb{Z}^n$ to $A$ is called a \emph{$n$-dimensional word on the alphabet $A$}, or a \emph{configuration}. The elements of $A$ are called \emph{letters}. Let $\mbf{w}$ be an $n$-dimensional word and $u$ be a vector in $\bb{Z}^n$. If $\mbf{w}(x) = \mbf{w}(x + u)$ for each $x \in \bb{Z}^n$, then $\mbf{w}$ is called \emph{$u$-periodic}. A word is \emph{periodic} if there exists a vector $u$ such that the word $u$-periodic. A word is \emph{fully periodic} if there exist $n$ linearly independent vectors $u_1, u_2, \ldots, u_n \in \bb{Z}^n$ such that the word is $u_i$-periodic for all $i = 1, 2, \ldots n$.

Two finite words $u$ and $v$ are said to be \emph{abelian equivalent}, denoted by $u \sim_{ab} v$, if and only if $|u|_a = |v|_a$ for each $a \in A$, where $|u|_a$ denotes the number of occurrences of the letter $a$ in $u$. It is readily verified that $\sim_{ab}$ defines an equivalence relation on the set of finite words. 

Let $\mbf{w}$ be an $n$-dimensional word. Its \emph{complexity} is defined as a function $p_{\mbf{w}}:\mathbb{N}^n\to \mathbb{N}$ counting for each $(m_1,\dots,m_n)$ the number of distinct $m_1\times\dots \times m_n$-blocks (or factors) of $\mbf{w}$.  Similarly, its \emph{abelian complexity} is defined as a function counting the number of abelian classes of $n_1\times\dots \times n_d$-factors of $\mbf{w}$.

A finite subset of $\bb{Z}^n$ is called \textit{a figure of $\bb{Z}^n$}. Let $F_1$ and $F_2$	be figures of $\bb{Z}^n$. If there exists a translation $\tau$ such that $\tau(F_1) = F_2$, then we say that $F_1$ and $F_2$ are \textit{equivalent} and write $F_1 \sim F_2$. An equivalence class under $\sim$ is called a \textit{pattern}.

Let $F$ be a figure of $\bb{Z}^n$. A \textit{figure with weights} in $\bb{Z}^n$ is a finite set 
\[F^w = \{(u, g_u) | u \in F, g_u \in \bb{Z}\},\] 
where the integer $g_u$ is called the weight of a point $u$. Let $F^w_1$ and $F^w_2$ be figures with weights in $\bb{Z}^n$. If there exists a translation $\tau$ such that 
		\[
			F^w_2=\{(\tau(u), g_u)| (u,g_u) \in F^w_1\},
		\]
then we say that $F_1$ and $F_2$ are \textit{equivalent} and write $F_1 \sim_g F_2$. An equivalence class under $\sim_g$ is called  \textit{a pattern with weights}.

	
  
A natural generalization of the notion of a  complexity is the pattern complexity. Let $F^w=\{  (u_1, g_1), \ldots, (u_l,g_l)  \}$ 	be a figure with weights, and $A = \{a_1, a_2, \ldots, a_k\}$ be an alphabet. {Consider the set 
$\bb{Z}[a_1, a_2, \ldots, a_k]$ of linear combinations over variables from $A$.} A linear polynomial 
\[\sum\limits_{i = 1}^l \mbf{w}(u_i)g_i \in \bb{Z}[a_1, a_2 \ldots a_k]\] 
is called a \emph{linear combination of the word $\mbf{w}$ over $F^w$}. 
We say that this linear combination 
is a \emph{$\cal{P}$-linear combination} if $F^w$ belongs to the pattern (with weights) $\cal{P}$. Essentially, the $\cal{P}$-linear combination gives the sum of the letters from $F$ multiplied by  their weights.
	
	The \emph{abelian pattern complexity}  $a_{\mbf{w}}(\cal{P})$ of $\mbf{w}$ is the function counting the number of distinct $\cal{P}$-linear combinations of $\mbf{w}$. If all the weights are equal to 1, then we have the abelian pattern complexity counting the number of abelian equivalence classes of $\cal{P}$-factors, 
	i.e., factors of $\mbf{w}$ restricted to figures from $\cal{P}$. When we have weights, the occurrences of  each letter are counted with corresponding multiplicities (which could also be negative).

	For convenience, we select one figure to associate with a pattern $\cal{P}$ and call it canonical:
	
		\begin{definition}
	    Let $\cal{P}$ be a pattern with weights in the  $n$-dimensional space and $F_{\cal{P}}$ be the figure of $\cal{P}$ such that all the coordinates of the points in $F_{\cal{P}}$ are non-negative, and for each $i \in \{1, 2, \ldots ,n\}$ there exists a point of $F_{\cal{P}}$ with $i$-th coordinate equal to $0$. The figure $F_{\cal{P}}$ is called \emph{the  canonical figure} of the pattern $\cal{P}$.
	\end{definition}

  For example, the canonical figure for the pattern with weights defined by the figure $R$ from Fig.~\ref{fig:example_of_char}, is the Figure $R$ shifted by $(-1,-1)$, i.e. $\{((0,0), 10); ((0,1), 2); ((0, 3), -3); ((1,0), 1) ; ((1, 2), 1) ; ((2, 1), 4)\}$. %
	
\begin{definition}\label{def_poly} The \emph{polynomial of a pattern} $\cal{P}$ is defined by
	$$
	    Poly_{\cal{P}}(x) = \sum_{(t, g_t) \in F_{\cal{P}}} g_t x^t,
	$$
where $x^t = x_1^{t_1} x_2^{t_2} \cdots x_n^{t_n}$. 	\end{definition}
	
We remark that there is no particular difference which figure to choose for these notions. We could also choose another figure or associate with a pattern a class of polynomials for all figures from the pattern; these polynomials differ up to a multiple of the form $x_1^{i_1} \cdots x_n^{i_n}$, corresponding to a translation by a vector $(i_1, \ldots, i_n)$. 
		
			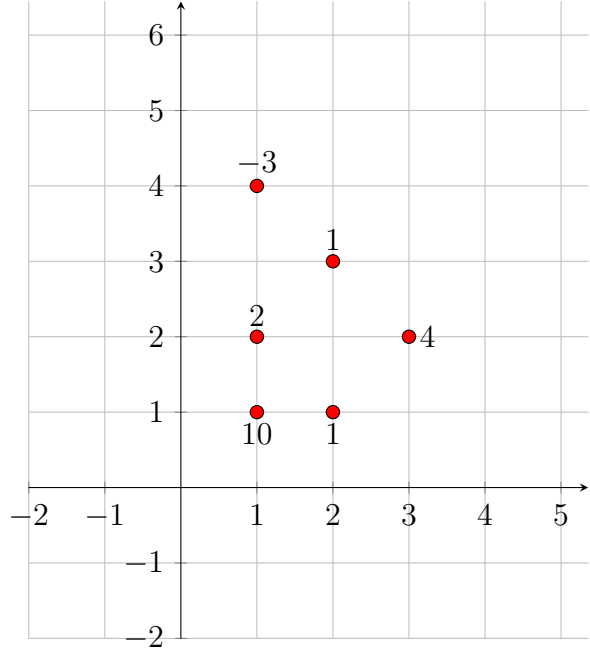
\begin{figure}[h]
	\centering
	\begin{tikzpicture}[line cap=round,line join=round,>=triangle 45,x=1cm,y=1cm]
\begin{axis}[
x=1cm,y=1cm,
axis lines=middle,
ymajorgrids=true,
xmajorgrids=true,
xmin=-2,
xmax=5.356501314907437,
ymin=-2,
ymax=6.446419066605355,
xtick={-8,-7,...,10},
ytick={-5,-4,...,6},]
\clip(-2,-2) rectangle (5.356501314907437,6.446419066605355);
\begin{scriptsize}
\draw [fill=red] (1,2) circle (2.5pt) node[above] {$2$};
\draw [fill=red] (1,4) circle (2.5pt) node[above] {$-3$};
\draw [fill=red] (2,3) circle (2.5pt) node[above] {$1$};
\draw [fill=red] (3,2) circle (2.5pt) node[right] {$4$};
\draw [fill=red] (2,1) circle (2.5pt) node[below] {$1$};
\draw [fill=red] (1,1) circle (2.5pt) node[below] {$10$};
\end{scriptsize}
\end{axis}
\end{tikzpicture}
	\caption{Figure $R = \{((1,1), 10); ((1,2), 2); ((1, 4), -3); ((2,1), 1) ; ((2, 3), 1) ; ((3, 2), 4)\}$} \label{fig:example_of_char}
	\end{figure}
	
	For example, the polynomial of the pattern $\cal{R}$ of the figure $R$ from Fig.~\ref{fig:example_of_char} is given by the following:
	$$
		Poly_{\cal{R}} (x, y) = 10 + 2y - 3y^3 + x + 4xy^2 + x^2y.
	$$


In this paper we are interested in patterns for which there exist non-periodic words with abelian pattern complexity equal to 1, and we introduce the following notion:
	\begin{definition} \label{def:good_patterns}
		A pattern $\cal{P}$ is called \emph{abelian rigid} if $a_{\bf{w}}(\cal{P}) = 1$ only for fully periodic words $\bf{w}$.
	\end{definition}

 In a symbolic dynamical terminology, the set of words with pattern abelian complexity equal to 1 over a pattern $\cal{P}$ forms a subshift. If we endow the alphabet $A$ with the discrete topology, $A^{\mathbb{Z}^d}$ with the product topology, and define a $\mathbb{Z}^d$-action on $X = A^{\mathbb{Z}^d}$ by $(T^u \mbf{w})(x) = \mbf{w}(x + u)$ for $u \in \mathbb{Z}^d$, then a {\emph{subshift}} on $A$ is a pair $
(X, T)$, where $X$ is a closed and $T$-invariant subset of $A^{\mathbb{N}^d}$. A subshift is said to be \emph{of finite type} if it is defined by finitely many forbidden patterns. It is not hard to see that the set of words with pattern abelian complexity equal to 1 over a pattern $\cal{P}$ form a subshift of finite type. Indeed, one can define the set of forbidden factors as all factors containing two distinct abelian  $\cal{P}$-factors at distance 1 horizontally or vertically.
	
	\section{Abelian complexity and one-dimensional words} \label{section:1d}

In this section, all words are one-dimensional. As the main result of this section, we give a necessary condition on a pattern $\cal{P}$ for existence of a non-unary word with pattern abelian complexity over $\cal{P}$ equal to 1 in terms of cyclotomic polynomials. We prove the condition in a slightly more general form, for patterns with weights and for integer-valued functions with constant sums. Moreover, for the sums the condition is necessary and sufficient (see  Theorem \ref{th:1d}).
	 
We begin with some notation.  Let $\mbf{w}$ be a bi-infinite one-dimensional word, i.e., $\mbf{w}$ belongs to $A^{\bb{Z}}$ for some alphabet $A$. The word $\mbf{w}$ is said to be \emph{periodic} if there exists an integer $p$ such that $\mbf{w}(i + p) = \mbf{w}(a)$ for each integer $i$.

	As we defined earlier, we associate with a pattern $\cal{P}$  its canonical figure $F_{\cal{P}}$ and a polynomial $Poly_{\cal{P}}(x) $ (see Definition \ref{def_poly}). In the one-dimensional case, $Poly_{\cal{P}}(x) $  corresponds to the element $F_{\cal{P}}$ of  $\cal{P}$ with its leftmost element at $0$. We then let $d$ denote its rightmost element:
		$$d= \max\limits_{(u, g_u) \in F_{\cal{P}}}\{u | g_u \ne 0\}. $$
		We say that  $d + 1$ is the \textit{diameter} of the pattern $\cal{P}$. 
	
In the same way we can define the figure of a polynomial $P(x)$. Let $P (x) = \sum_{i = 0}^m g_i x^i$ be a polynomial. Then the \emph{figure} $Fig_{P}$ of the polynomial $P(x)$ is defined by 
	\[
		Fig_{P} = \{(i, g_i) | i = 0, 1, \ldots, m\}.
	\]

In the next section, we are going to study the main question of the paper: characterize two-dimenisonal  patterns for which the only words functions with constant sums (or words with abelian complexity equal to 1), are the periodic ones. The following proposition treats a similar question in the one-dimensional case. In fact, this is easy to see that that all one-dimensional patterns satisfy the above property:

\begin{prop} \label{lem:always_per}
		Let $\cal{P}$ be a pattern with weights and $\mbf{w}$ be a word on an alphabet which is a finite subset of complex numbers. Suppose that there exists a constant $C$ such that 
		\[
		\sum\limits_{x \in F} \mbf{w}(x) g_x = C
		\]
		for any $F \in \cal{P}$. Then $\mbf{w}$ is periodic. 
	\end{prop}
		
\begin{proof}
Let $d+1$ be the diameter of $\cal{P}$. We claim that there exist two numbers $m, n$ such that $\mbf{w}(n + i) = \mbf{w}(m+ i)$ for each $0 \leq i \leq d$. Indeed, $p_{\mbf{w}}(d+1) \leq |A|^{d+1}$, i.e., there exists only $|A|^{d+1}$ distinct factors of length $d+1$, hence there must be two (and in fact infinitely many) equal factors of length $d+1$. Consider the word $\mbf{s}$ defined by $\mbf{s}(k) = \mbf{w}(k + n) - \mbf{w}(k + m)$. Let us prove that $\mbf{s}$ is an all-$0$ word. Assume the converse, then there exists $i$ such that $\mbf{s}(i) \ne 0$. We let $r$ denote the minimal (by absolute value) position such that $\mbf{s}(r) \ne 0$. Suppose $r < 0$; then 	$\sum\limits_{(x, g_x) \in F_{\cal{P}}} \mbf{s}(x + r) g_x = 0$, where $F_{\cal{P}}$ is the canonical figure of the pattern $\cal{P}$. By assumption, we obtain $\mbf{s}(r) g_0 = 0$ and $\mbf{s}(r) \ne 0$. Hence $g_0=0$; a contradiction. So, $\mbf{s}$ is an all-$0$ word and $\mbf{w}(k + n) = \mbf{w}(k + m)$ for each $k \in \bb{Z}$, i.e., $\mbf{w}$ is $|n - m|$-periodic.  
	\end{proof}
	

A similar statement for words with pattern abelian complexity 1 is a direct corollary from the above proposition:
	\begin{corollary} \label{cor:always_per_words}
	Let $\cal{P}$ be a pattern with weights and $\mbf{w}$ be a word such that $a_{\mbf{w}}(\cal{P})=1$. Then $\mbf{w}$ is periodic. 
	\end{corollary}

			We recall that the $n$'th cyclotomic polynomial is defined by 
	\[ \Phi_n(x) = \prod_{1 \leq k \leq n; (k, n) = 1} (x - e^{2\pi i \frac k n}).\]

The following theorem gives a necessary and sufficient condition for existence of a non-trivial word (i.e., consisting not only of $0$'s) with constant $\cal{P}$-sum: 
	
	\begin{theorem} \label{th:1d}
		Let $\cal{P}$ be a pattern with weights, and let the alphabet $A$ be a finite subset of complex numbers. 
		Then there exists a non-zero word $\mbf{w}$ such that
		\begin{equation}\label{eq:fix_sum}
			\sum\limits_{x \in F} \mbf{w}(x)g_x = 0
		\mbox{ for all } F \in \cal{P} 
		\end{equation}
	 if and only if there exists  $n$ such that $Poly_{\cal{P}}(x)$ is divisible by $\Phi_n(x)$  in the ring $\bb{Q}[x]$.
	\end{theorem}
		
	To prove Theorem \ref{th:1d}, we need an auxiliary lemma:	
	
	\begin{lemma} \label{lem:lin_comb}
		Let $\cal{P}_1, \cal{P}_2, \ldots, \cal{P}_k$ be patterns with weights and $Z_1(x), Z_2(x), \ldots, Z_k(x)$ be some polynomials over $\bb{Z}[x]$. 		Suppose that
			$$\sum\limits_{(u, g_u) \in F} g_u \mbf{w}(u) = 0$$
		{for each $i$} and each $F \in \cal{P}_i$. Then  for the polynomial $P (x) = \sum_{i =1}^k Z_i(x) Poly_{\cal{P}_i}(x)$ we have   
			\[ \sum_{(u, g_u) \in Fig_{P}} g_u \mbf{w}(u)= 0.\]
		\end{lemma}

	This lemma is straightforward since the figure $Fig_{x^l Poly_{\cal{P}}(x)} $ is a translation of the characteristic figure of a pattern $\cal{P}$, and 
	$$
		\sum\limits_{(u, g_u) \in F_{\cal{P}}} g_u \mbf{w}(u) + 
		\sum\limits_{(u, g_u) \in F_{\cal{P}'}} g_u \mbf{w}(u) = 
		\sum\limits_{(u, g_u) \in Fig_{Poly_{\cal{P}} + Poly_{\cal{P}'}}} g_u \mbf{w}(u),
	$$
	where $F_{\cal{P}}$ and $F_{\cal{P}'}$ are canonical figures of patterns $\cal{P}$ and $\cal{P}'$.
	
	
	\begin{proof}[Proof of Theorem \ref{th:1d}]
		First, we prove that if there exist a non-zero word with sums in $\cal{P}$-factors equal to 0, then $Poly_{\cal{P}}$ is divisible by some cyclotomic polynomial.
		
		Let $\mbf{w}$ be a word such that $\sum\limits_{x \in F} \mbf{w}(x) g_x = 0$ for each $F \in \cal{P}$. By Proposition \ref{lem:always_per}, the word $\mbf{w}$ is periodic. Suppose that $n$ is the period of $\mbf{w}$ and that $Poly_{\cal{P}}(x)$ and $x^n - 1$ are coprime in the ring $\bb{Q}[x]$. 
		Then there exist polynomials $Q_1 (x), Q_2 (x) \in \bb{Q}[x]$ such that 
		\[Poly_{\cal{P}}(x) Q_1 (x) +  (x^n -1) Q_2 (x) = 1.\] 
		Then there exist $Z_1(x), Z_2(x) \in \bb{Z}[x]$ such that 
		\[Poly_{\cal{P}}(x) Z_1(x) + Z_2(x) (x^n -1) = C\]
		for some integer $C\neq 0$. 
		Since $n$ is the period of the word $\mbf{w}$, we have that 
		\[\sum_{(u , g_u) \in F} g_u \mbf{w}(u) = 0\]  
		for each $F$ in the pattern defined by $x^n -1$.
		From Lemma \ref{lem:lin_comb} it follows that 
		\[ \sum_{(u , g_u) \in Fig_{Poly_{\cal{P}}(x) Z_1(x) + Z_2(x) (x^n -1) }} g_u \mbf{w}(u) = 0,\] 
		i.e., $C\mbf{w}(0) = 0$  and so $\mbf{w}(0) = 0$. With the same reasoning, we get that  $\mbf{w}(i) = 0$ for all $i$. 
		
		So, we proved that $x^n -1$ and  $Poly_{\cal{P}}$ are not coprime. Then $Poly_{\cal{P}}$ is divisible by $\Phi_d$ for some $d$ such that $n$ is divisible by $d$.
		
Now, let us construct for each pattern $\cal{P}$ with a polynomial  divisible by $\Phi_n$ a non-zero function for which its sums in all $\cal{P}$-factors are equal to 0. 
Suppose first that we have an equality: $Poly_{\cal{P}}=\Phi_n$.  
Let us construct a word $\mbf{s}$ such that $\sum\limits_{(x, g_x) \in F} \mbf{s}(x)g_x = 0$ for each figure $F \in \cal{P}$. We set $\mbf{s}(x) = e^{\frac {2\pi i x}  n}$. Then, expressing $F$ as $v+F_{\cal{P}}$, we get 
		$$\sum_{(x, g_x) \in F_{\cal{P}}} g_x\mbf{s}(x + v) = e^{\frac {2 \pi i v} n} \Phi_n\left(e^{\frac {2\pi i}  n}\right) = 0.$$
		So, the ``only if'' part is proved for $Poly_{\cal{P}}=\Phi_n$. 
		Now, consider the general case of a pattern $\cal P$ such that $\Phi_n$ divides $Poly_{\cal{P}}$. By Lemma~\ref{lem:lin_comb}, the above equiality holds true also for this case, i.e., 
		$$\sum_{(x, g_x) \in F_{\cal{P}}} g_x\mbf{s}(x + v) = 0.$$
		This concludes the proof.

		
	\end{proof}

As a corollary, we get a necessary condition for a pattern to be abelian rigid:	
	\begin{corollary} If for a one-dimensional pattern $\cal{P}$ with weights there exists a non-unary word with pattern abelian complexity over $\cal{P}$ equal to $1$, then  there exists  $n$ such that $Poly_{\cal{P}}(x)$ is divisible by $\Phi_n(x)$  in the ring $\bb{Q}[x]$.
	\end{corollary}
	
\begin{remark}	This condition is not sufficient for existence of a non-unary word with pattern abelian complexity over $\cal{P}$ equal to $1$ though. For example, for a pattern corresponding $\Phi_6=x^2-x+1$ there is no word with abelian complexity equal to 1 over this pattern. To see it, it is enough to check words up to length $4$.
\end{remark}	


	\section{Pattern abelian complexity of two-dimensional words by convex patterns} \label{section:2d} 
	
	Throughout this section $\bf{w}$ denotes a two-dimensional word.
	
	\begin{definition} \label{def:convex_pattern}
		A figure  $F$ (without weights) is called \emph{convex} if 
		$F = conv_{\bb{R}^2}(F) \cap \bb{Z}^2$, where $conv_{\bb{R}^2}(F)$ is convex hull of the set $F$ in two-dimensional Euclidean space.
				A pattern $\cal{P}$ is called \emph{convex} if all figures of this pattern are convex. 
	\end{definition}	
	Clearly, the definition is correct, since all figures of a pattern are translations of one another, so if one of them is convex, then all of them are.
		
	

	
	 Let $v \in \bb{Z}^2$ be a vector and $n$ be a natural number. The polynomial $l(v, n) = \sum_{i = 0}^n x^{iv}$ is called \emph{strongly linear}. Now we can state our main result.
	
	\begin{theorem} \label{th:main_convex}
		Let $\cal{P}$ be a convex pattern and $F$ be a figure of this pattern. Then $char(F)$ has a non-trivial strongly linear divisor if and only if $\cal{P}$ is not abelian rigid. 
	\end{theorem}

In the next section, we give an alternative statement of the theorem in geometric form (see Theorem \ref{th:41}').

	
	One direction of the statement is easy, and we prove it right away in a stronger form (Theorem \ref{th:has_divisor_then_not_peridic}). For the other direction, we would need several additional lemmas and definitions, so we prove it in the consequent subsections. The other direction is also stated in a stronger form (see Theorem \ref{th:sum_then_per})

The following theorem is a stronger version of the only if part of Theorem \ref{th:main_convex} for patterns with weights:

	\begin{theorem} \label{th:has_divisor_then_not_peridic}
	Let $\cal{P}$ be a pattern with weights. 
	If $Poly_{\cal{P}}$ has a non-trivial strongly linear divisor, then $\cal{P}$ is not abelian rigid. 
	\end{theorem}

	In other words, there exists a two-dimensional word $\bf{w}$ such that $\bf{w}$ is not fully periodic and $a_{\bf{w}}(\cal{P}) = 1$.
	
	\begin{proof}
		Suppose first that  $Poly_{\cal{P}}$  has a strongly linear divisor, i.e., $Poly_{\cal{P}}  = l(v, n)Q$ for some integer $n$, a vector $v$ and a polynomial $Q$. We let $\cal{P}'$ denote the pattern corresponding to the polynomial $l(v,n)$, so that $Poly_{\cal{P}'} = l(v,n)$.
		
		Let $u, u'$ be a pair of basis vectors such that $v = ku$ for some integer $k$.		
		We now construct a word $\bf{w}$ on the alphabet $\{0, 1, 2 \ldots n - 1\}$ such that $a_{\bf{w}}(\cal{P}') = 1$. We set
		\begin{eqnarray}\label{eq:example_has_div}
			{\bf{w}} ((ka + r)u + bu') = 
			\begin{cases}
				(a + b) \mod n, & \text{ if $b$ is not prime}, \\
				(a + b + 1) \mod n, & \text{ if $b$ is prime};
			\end{cases}
		\end{eqnarray}		 
			for all $r < k$. Obviously, $\bf{w}$ has only one abelian class in $\cal{P}$ and it is not fully periodic.
	\end{proof}
	
	\begin{remark}
	It is not hard to see that Theorem \ref{th:has_divisor_then_not_peridic} can easily be generalized to higher dimensions. In addition to strongly linear divisors, one can consider divisors in subspaces of higher dimensions: planes, hyperplanes, etc. which allow to get aperiodic words and functions.
	\end{remark}

	
	\subsection{Relatively prime patterns and their coordinates}
	
	Let $u, v$ be two basis vectors in $\bb{Z}^2$, i.e., for any vector $z \in \bb{Z}^2$ there exist integers $a$ and $b$ such that $z = au + bv$. We will say that $(a, b)_{u, v}$ are \emph{$(u,v)$-coordinates} of the point $z$.
	
	
	For integers $l_0, l_1 , \ldots , l_n, r_0, r_1, \ldots , r_n$ such that 
	$l_0 < r_0, l_1 \leq r_1, l_2 \leq r_2, \ldots , l_{n-1} \leq r_{n-1}, l_n < r_n$, 
	we define the figure $F_{u, v}(l_0, l_2, \ldots l_n, r_0, r_2, \ldots r_n)$ as follows:
	$$
		F_{u, v}(l_0, l_2, \ldots, l_n; r_0, r_2, \ldots, r_n) = 
		\{(i, j)_{u, v} | 0 \leq i \leq n, l_i \leq j < r_i\}.
	$$
	We let $\cal{P}_{u, v}(l_0, l_2, \ldots, l_n; r_0, r_2, \ldots, r_n)$ denote the pattern of $F_{u, v}(l_0, l_2, \ldots, l_n; r_0, r_2, \ldots, r_n)$.
		
	\begin{prop} \label{prop:corr}
		Let $\cal{P}$ be a convex pattern and $u, v$ be basis vectors. Then there exist integers $n$ and $l_0, l_1, \ldots, l_n, r_0, r_1, \ldots, r_n$ such that 
		\begin{equation} \label{eq:basis_repr}
			\cal{P} = \cal{P}_{u, v}(l_0, l_1, \ldots, l_n; r_0, r_1, \ldots, r_n).
		\end{equation}
	\end{prop}
	\begin{proof}
		Suppose that a figure $F$ belongs to the pattern $\cal{P}$ and $q$ is the minimal first coordinate of the figure $F$. Then the figure $F'=F - qu  = \{x - qu | x \in F\}$ also belongs to the pattern $\cal{P}$.
		
		Suppose that the maximal first coordinate of $F'$ is $n$. Let $C_i$ be the set of points of the figure $F'$ such that their first coordinate is equal to $i$.
		
		Put $l_i$ as the minimal second coordinate of the points in the set $C_i$ and $r_i$ as the maximal second coordinate plus $1$. 
		The figure $F'$ is convex, so 
		$C_i = \{(i, t) | l_i \leq t < r_i\}$. Then 
		\[ F' = F_{u, v}(l_0, l_2, \ldots, l_n; r_0, r_2, \ldots, r_n).\] 
		and $\cal{P}$ is the pattern of the figure $F'$.
	\end{proof}
	
	 We shall call this form of a convex pattern a $(u, v)$-\emph{representation}.
	
	
		
	


	  
We now introduce the concept of a relatively prime pattern, which as we will later show gives a geometric interpretation of the characterization of abelian rigid patterns.	  
Let $u$ be an integer vector and $\cal{P}$ be a convex pattern. Each line parallel to $u$ intersects the pattern in several integer points; we call it the length of the intersection. Now consider all lines that are parallel to the vector $u$ and intersect some figure of $\cal{P}$ (say, the canonical one). If  the greatest common divisor of the lengths of intersections of the set of such lines is equal to 1, then the pattern $\cal{P}$ is called $u$-relatively prime. A convex pattern  $\cal{P}$ is called  \emph{relatively prime} if it is $u$-relatively prime for each vector $u$. On Fig.~\ref{fig:rel_prime} 
one can see an example of a relatively prime pattern and an example of non-relatively prime pattern.

\begin{figure}[h]
    \centering
       
	\begin{tikzpicture}
	
	\begin{scope}[line cap=round,line join=round,>=triangle 45,x=1cm,y=1cm, scale =0.9]
\draw [color=cqcqcq,, xstep=1cm,ystep=1cm] (-9.5,-0.5) grid (-1.5,3.5);
\clip(-9.5,-0.5) rectangle (-1.5,3.5);
\draw [line width=2pt,domain=-10.020336072162362:-0.6981087197222217] plot(\x,{(--8--1*\x)/2});
\draw [line width=2pt,domain=-10.020336072162362:-0.6981087197222217] plot(\x,{(--9--1*\x)/2});
\draw [line width=2pt,domain=-10.020336072162362:-0.6981087197222217] plot(\x,{(--7--1*\x)/2});
\draw [line width=2pt,domain=-10.020336072162362:-0.6981087197222217] plot(\x,{(--6--1*\x)/2});
\draw [line width=2pt,domain=-10.020336072162362:-0.6981087197222217] plot(\x,{(--5--1*\x)/2});
\begin{scriptsize}
\draw [fill=ududff] (-7,1) circle (2.5pt);
\draw [fill=ududff] (-5,2) circle (2.5pt);
\draw [fill=ududff] (-7,0) circle (2.5pt);
\draw [fill=ududff] (-6,1) circle (2.5pt);
\draw [fill=ududff] (-5,1) circle (2.5pt);
\draw [fill=ududff] (-4,2) circle (2.5pt);
\draw [fill=ududff] (-6,0) circle (2.5pt);
\draw [fill=ududff] (-5,0) circle (2.5pt);
\draw [fill=ududff] (-8,0) circle (2.5pt);
\draw [fill=xdxdff] (-9,0) circle (2.5pt);
\draw [fill=xdxdff] (-4,1) circle (2.5pt);
\draw [fill=xdxdff] (-3,1) circle (2.5pt);
\draw [fill=xdxdff] (-3,2) circle (2.5pt);
\draw [fill=xdxdff] (-3,3) circle (2.5pt);
\end{scriptsize}
\end{scope}
\begin{scope}[line cap=round,line join=round,>=triangle 45,x=1cm,y=1cm, yshift = 0.5cm, xshift = 3cm, scale = 0.7]
\draw [color=cqcqcq,, xstep=1cm,ystep=1cm] (-5.5,-1.5) grid (2.5,4.5);
\clip(-5.5,-1.5) rectangle (2.5,4.5);
\draw [line width=2pt,domain=-15.36:15.36] plot(\x,{(--3-0*\x)/1});
\draw [line width=2pt,domain=-15.36:15.36] plot(\x,{(--6-0*\x)/3});
\draw [line width=2pt,domain=-15.36:15.36] plot(\x,{(--1-0*\x)/1});
\draw [line width=2pt,domain=-15.36:15.36] plot(\x,{(-0-0*\x)/1});
\begin{scriptsize}
\draw [fill=ududff] (-2,3) circle (4pt);
\draw [fill=ududff] (-1,3) circle (4pt);
\draw [fill=ududff] (-3,2) circle (4pt);
\draw [fill=ududff] (-2,2) circle (4pt);
\draw [fill=ududff] (-1,2) circle (4pt);
\draw [fill=ududff] (0,2) circle (4pt);
\draw [fill=ududff] (-4,1) circle (4pt);
\draw [fill=ududff] (-3,1) circle (4pt);
\draw [fill=ududff] (-2,1) circle (4pt);
\draw [fill=ududff] (-1,1) circle (4pt);
\draw [fill=ududff] (0,1) circle (4pt);
\draw [fill=ududff] (1,1) circle (4pt);
\draw [fill=ududff] (-2,0) circle (4pt);
\draw [fill=ududff] (-3,0) circle (4pt);
\end{scriptsize}
\end{scope}
	
	\end{tikzpicture}
 \caption{Examples of a $(2,1)$-relatively prime (on the left) and not relatively prime (not $(1, 0)$-relatively prime) pattern (on the right)}
    \label{fig:rel_prime}
\end{figure}
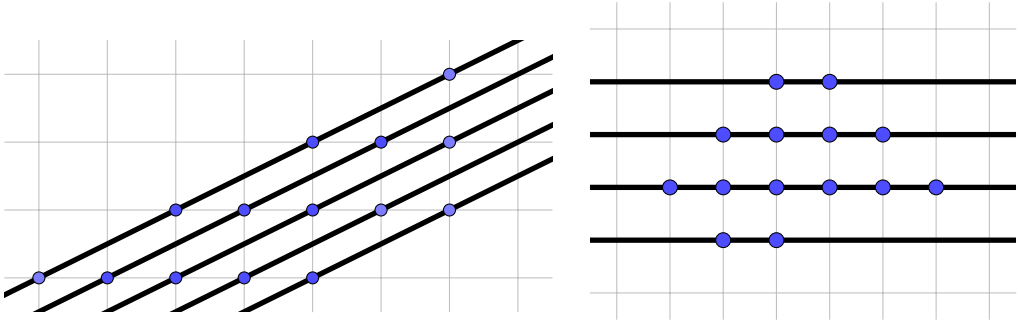
	
	A more precise but a bit less intuitive definition of a relatively prime pattern can be also given in terms of $(u,v)$-representations. 
		    Let $u$ be an integer vector and $v$ be an integer vector such that $\{u, v\}$ is basis. 
	    Let $\cal{P}$ be a convex pattern and $\cal{P}_{u, v}(l_0, l_1, \ldots, l_n; r_0, r_1, \ldots ,r_n)$ be a $(u, v)$-representation of $\cal{P}$. The pattern $\cal{P}$ is called \emph{$u$-relatively prime pattern} if $gcd(r_0-l_0, r_1 - l_1 \ldots , r_n - l_n) = 1$. 

	An alternative geometric formulation of Theorem \ref{th:main_convex} via the notion of a relative pattern is as follows:
		\begin{theorem} \label{th:41}
		Let $\cal{P}$ be a convex pattern.
		Then $\cal{P}$ is relatively prime if and only if $\cal{P}$ is  abelian rigid. 
	\end{theorem}
	
	The statement is equivalent to Theorem \ref{th:main_convex} since being relatively prime is clearly the same thing as not being divisible by a strongly linear divisor.

\subsection{Proof of Theorem \ref{th:main_convex}}	\label{section:main_proof}
	
We are going to prove the theorem in the form of Theorem \ref{th:41}. By Theorem \ref{th:has_divisor_then_not_peridic}, an abelian rigid pattern must be relatively prime. So, it remains to prove that $\mbf{w}$ is fully periodic if $a_{\mbf{w}}(\cal{P}) = 1$ for some convex relatively prime pattern.
	
\begin{definition}\label{def:par&ext_fig}
		Let $D_1$ and $D_2$ be two convex polygons with vertices $a_1, a_2, \ldots , a_n$ and $b_1, b_2, \ldots, b_m$, respectively. 
		The polygons $D_1$ and $D_2$ are called \emph{parallel} if \linebreak 
		${n = m}$ and all their corresponding edges are parallel, i.e., there exists $k$ such that $a_{i \mod n}a_{i + 1 \mod n} $ is parallel to $b_{i + k \mod n}b_{i + k + 1 \mod n}$ for each $i$.
\end{definition}	
	
	The following proposition constitutes the main part of the proof of Theorem \ref{th:main_convex}.  Essentially, it says that if a word is equal to $0$ in a big enough polygon $D$ parallel to a convex hull $D_0$ of $\cal{P}$, then it also equals $0$ in an extended polygon $D_1$ homothetic to the initial polygon $D_0$.  For a set $X$, we let $diam(X)$ denote the diameter of $X$ in Euclidean metric on the plane. The accurate statement follows: 
	
	
	\begin{prop} \label{lem:full_ext}
		Let $A$ be a finite subset of complex numbers containing
		$0$. Let $\cal{P}$ be a relatively prime pattern and $\bf{w}$ be a two-dimensional word on $A$ such that 
		\begin{equation} \label{eq:sum0_z2}
			\sum_{x \in F} {\bf{w}}(x) = 0
		\end{equation}
		for all figures $F \in \cal{P}$. 
			Let $D$ be the convex hull of the figure $F_{\cal{P}}$. 
		Then there exists a number $N(|A|, \cal{P})$ such that the following implication holds: 
		
		If there exists a polygon $D_0$ parallel to $D$ such that 
		\begin{itemize}
			\item any edge of $D_0$ is longer than $N(|A|, \cal{P})$;
			\item ${\bf{w}}(x) = 0$ for each $x \in D_0 \cap \bb{Z}^2$,
		\end{itemize}
		then there exists a polygon $D_1$ parallel to $D$  such that 
		\begin{itemize}
			\item there exists a homothety $H$ with coefficient bigger than 1 such that $H(D_0) = D_1$;
			\item $diam(D_1) = diam(D_0) + 1$,  
			\item ${\bf{w}}(x) = 0$ for all $x \in D_1 \cap \bb{Z}^2$.
		\end{itemize}
	\end{prop}
	
	
	
	
	
	The proposition is illustrated on Fig. \ref{pic:full_ext}. 
	
	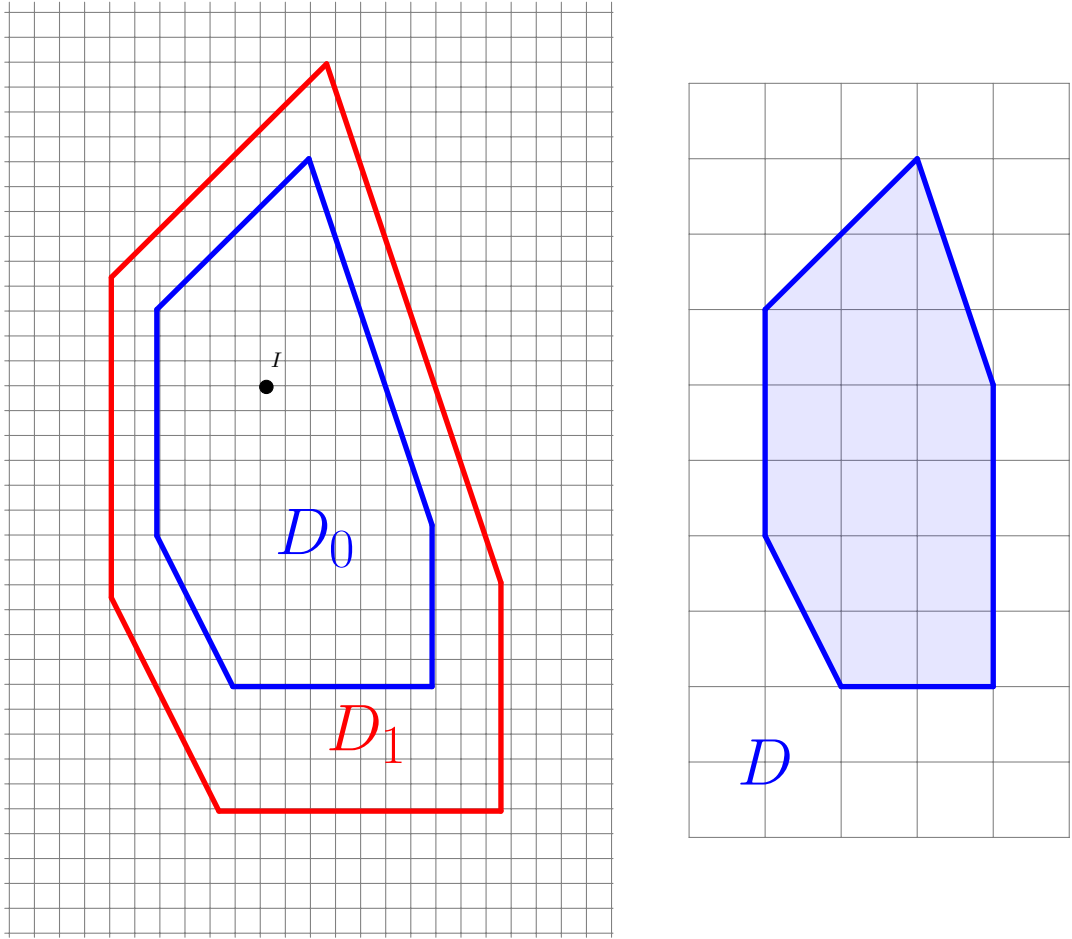
\begin{figure}[h]
	\begin{tikzpicture}
    \begin{scope}[line cap=round,line join=round,>=triangle 45,x=1cm,y=1cm]
\draw [color=gray, xstep=0.33cm,ystep=0.33cm] (-6,-6.324231865695288) grid (2,6.075725900116144);
\clip(-6,-6.324231865695288) rectangle (2,6.075725900116144);
\draw [line width=2pt, blue] (-2,4)-- (-4,2);
\draw [line width=2pt, blue] (-4,2)-- (-4,-1);
\draw [line width=2pt, blue] (-4,-1)-- (-3,-3);
\draw [line width=2pt, blue] (-2,4)-- (-0.38,-0.86);
\draw [line width=2pt, blue] (-3,-3)-- (-0.38,-3);
\draw [line width=2pt, blue] (-0.38,-0.86)-- (-0.38,-3);
\draw [line width=2pt, red] (0.5256620941068856,-1.6221014884977145)-- (-1.7675739341452374,5.257606596258655);
\draw [line width=2pt, red] (-1.7675739341452374,5.257606596258655)-- (-4.598729524579957,2.426451005823935);
\draw [line width=2pt, red] (-4.598729524579957,2.426451005823935)-- (-4.598729524579956,-1.8202823798281447);
\draw [line width=2pt, red] (-4.598729524579956,-1.8202823798281447)-- (-3.1831517293625966,-4.651437970262863);
\draw [line width=2pt, red] (-3.1831517293625966,-4.651437970262863)-- (0.5256620941068856,-4.651437970262864);
\draw [line width=2pt, red] (0.5256620941068856,-4.651437970262864)-- (0.5256620941068856,-1.6221014884977145);
\begin{scriptsize}
\draw [fill=black] (-2.5592841305036447,0.973835920177381) circle (2.5pt);
\draw (-2.4241347270615585,1.3370499419279875) node {$I$};
\draw[color=blue] (-1.9173244641537341,-1.0280646183085256) node {\Huge{$D_0$}};
\draw[color=red] (-1.241577446943302,-3.6296906345686897) node {\Huge{$D_1$}};
\end{scriptsize}
\end{scope}
    \begin{scope}[xshift = 8cm , line cap=round,line join=round,>=triangle 45,x=1cm,y=1cm]
\draw [color=gray,, xstep=1cm,ystep=1cm] (-5,-5) grid (0,5);
\clip(-5,-5) rectangle (0,5);
\fill[line width=2pt,color=blue,fill=blue,fill opacity=0.10000000149011612] (-4,-1) -- (-4,2) -- (-2,4) -- (-1,1) -- (-1,-3) -- (-3,-3) -- cycle;
\draw [line width=2pt,color=blue] (-4,-1)-- (-4,2);
\draw [line width=2pt,color=blue] (-4,2)-- (-2,4);
\draw [line width=2pt,color=blue] (-2,4)-- (-1,1);
\draw [line width=2pt,color=blue] (-1,1)-- (-1,-3);
\draw [line width=2pt,color=blue] (-1,-3)-- (-3,-3);
\draw [line width=2pt,color=blue] (-3,-3)-- (-4,-1);
\draw[blue] (-4, -4) node {\Huge{$D$}};
\begin{scriptsize}
\end{scriptsize}
\end{scope}
	\end{tikzpicture}
	\caption{Illustration to Proposition \ref{lem:full_ext}.} \label{pic:full_ext}
	\end{figure}
	
	\begin{definition}
		Let $\bf{l_0}$ be a line on the plane. 
		A line $\bf{l}$ is \emph{a neighbour} of the line $\bf{l_0}$ if the following three conditions hold: 
		\begin{itemize}
			\item $\bf{l}$ is an integer line, i.e., $\bf{l}$ contains at least two integer points;
			\item $\bf{l}$ and $\bf{l_0}$ are parallel;
			\item there are no integer points between $\bf{l}$ and $\bf{l_0}$.
		\end{itemize}		  
	\end{definition}
	
	For the proof of this proposition we make use of the following lemma:
	
	\begin{lemma} \label{lem:step_of_ext}
		Let $v_1, v_2, \ldots, v_n$ be the vertices of the polygon $D_0$  satisfying the conditions of Proposition \ref{lem:full_ext} and 
		$\mbf{l}$ be the neighbour line of $v_n v_1$ such that $\mbf{l} \cap D_0 = \emptyset$.  
		Let $v_1'$ be the intersection of the lines $\mbf{l}$ and $v_1v_2$, and 
		$v_n'$ be the intersection of the lines $\mbf{l}$ and $v_{n-1}v_n$.
		Let $D_0'$ be the polygon with vertices $v'_1v_2v_3 \ldots v_{n-1}v_n'$. 
		Then $\mbf{w}(x) = 0$ for each $x \in D_0' \cap \bb{Z}^2$.
	\end{lemma}
	
	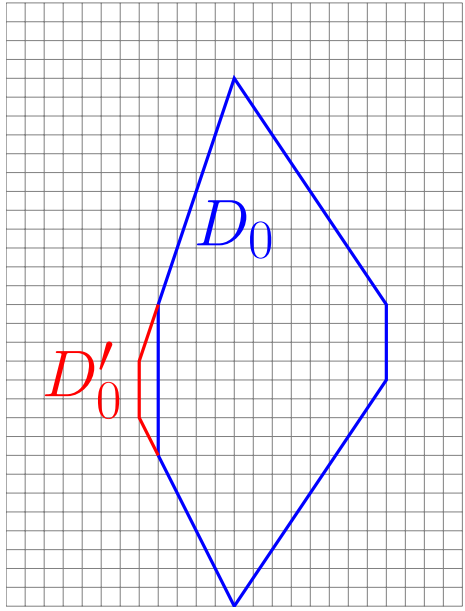
\begin{figure}[h]
	\centering
	\begin{tikzpicture}
        \clip (-3, -3) rectangle (3, 5);
        \draw[step = 0.25,very thin, gray] (-3, -3) grid (3, 5); 
        \draw[very thick, blue] (-1, -1) -- ++(0, 2) -- ++ (1, 3) -- ++(2, -3) -- ++(0, -1) -- ++(-2, -3) -- ++(-1, 2);
        \draw[blue] (0, 2) node {\Huge{$D_0$}};
        
        \draw[very thick, red] (-1, -1) -- ++(-0.25, 0.5) -- (-1.25,0.25) -- (-1, 1);
        \draw[red] (-2, 0) node {\Huge{$D_0'$}};
	\end{tikzpicture}
	\caption{Illustration for Lemma \ref{lem:step_of_ext}} \label{pic:step_of_ext}
	\end{figure}
	
	In other words, this lemma states that if $\mbf{w}$ is equal to $0$ in $D_0$, then we can extend $D_0$ by an integer line parallel to any its side, and  $\mbf{w}$ is also equal to $0$ in this extended polygon. We refer to Fig.~\ref{pic:step_of_ext} for the illustration.
	
	\begin{proof}[Proof of Lemma \ref{lem:step_of_ext}]
		Let us consider two cases. Case 1: $v_1v_n$ is an integer line, i.e., it contains at least two integer points. Case 2: $v_1v_n$ is not an integer line. 
		In the first case, we let $v_1''$ and $v_n''$ denote the points $v_1$ and $v_n$, and $\bf{l'} = \bf{l}$.  
		In the second case, put $v_1'' = {\bf{l'}} \cap v_1v_2$ and $v_n'' = {\bf{l'}}  \cap v_nv_{n-1}$, where $\bf{l'}$ is the other neighbour of line $v_1v_n$ (see Fig. \ref{pic:not_lines}).
		
		Let $(u,v)$ be the basis of $\bb{Z}^2$ such that the point $v_1'' - u$ belongs  $\bf{l}$ and $v= k v_1v_n$ for some positive real $k$. Let $\cal{P}(l_0, l_1, \ldots, l_m; r_0, \ldots ,r_m)$ be the $(u, v)$-representation of the pattern $\cal{P}$. We let $F_0$ denote the figure $F(l_0, l_1, \ldots, l_m; r_0, \ldots ,r_m)$. Without loss of generality we may assume that $l_0 = 0$, and that the {\color{blue}} second $(u, v)$-coordinate 
		of the segment $v_1''v_n''$ is equal to $-1$.
		Let ${\bf{l'}}_0 = {\bf{l}}, {\bf{l'}}_1, {\bf{l'}}_2, \ldots , {\bf{l'}}_m$ be the sequence of consecutive lines, i.e., ${\bf l'}_i$ and ${\bf l}'_{i+1}$ are neighbours for each $i = 0, 1, \ldots , m-1$ and ${\bf{l'}}_{i+1}\ne {\bf{l'}}_{i-1}$ (see Fig.~\ref{pic:not_lines}).
		Let $s_i$ be the intersection of the lines $v_1v_2$ and ${\bf l'}_i$, and $t_i$ be the intersection of $v_nv_{n-1}$ and ${\bf l'}_i$.
		Let $w_i$ be the sequence defined by $w_i(q) = \mbf{w}(r_q)$, where $r_q$ is $q$'th integer point on the segment $s_it_i$. 
		

	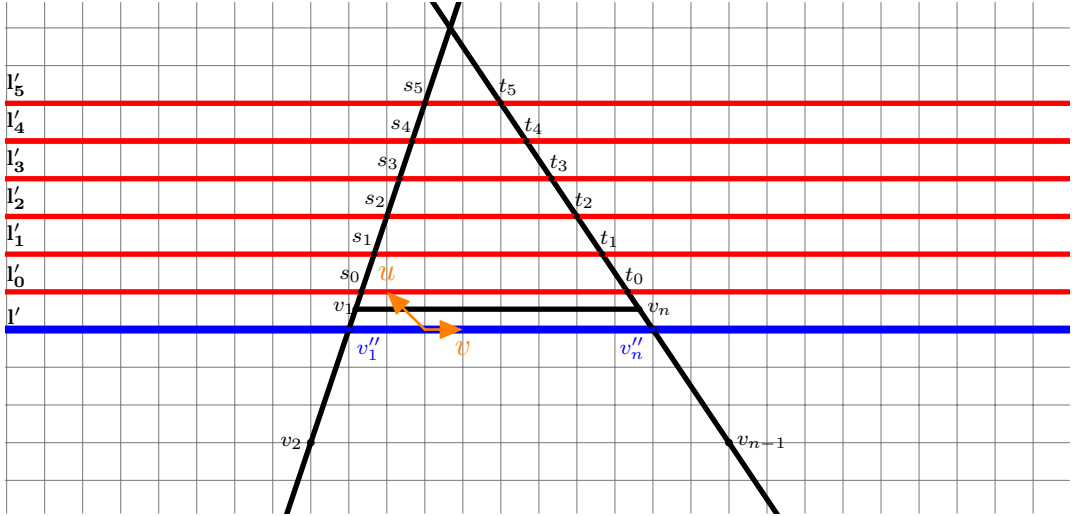
\begin{figure}[h]
	\centering
	\begin{tikzpicture}[line cap=round,line join=round,>=triangle 45,x=1cm,y=1cm, scale = 0.5]
\draw [color=gray,, xstep=1cm,ystep=1cm] (-14.036155147522754,-6.870345071077666) grid (13.957221388439766,6.675430820091583);
\clip(-14.036155147522754,-6.870345071077666) rectangle (13.957221388439766,6.675430820091583);

\draw [line width=3pt,blue, domain=-14.036155147522754:13.957221388439766] plot(\x,{(-16-0*\x)/8});
\draw [line width=2pt,domain=-14.036155147522754:13.957221388439766] plot(\x,{(-13-3*\x)/-1});
\draw [line width=2pt,domain=-14.036155147522754:13.957221388439766] plot(\x,{(--5-3*\x)/2});
\draw [line width=2pt, red,domain=-14.036155147522754:13.957221388439766] plot(\x,{(-8-0*\x)/8});
\draw [line width=2pt, red ,domain=-14.036155147522754:13.957221388439766] plot(\x,{(-0-0*\x)/8});
\draw [line width=2pt, red ,domain=-14.036155147522754:13.957221388439766] plot(\x,{(--8-0*\x)/8});
\draw [line width=2pt, red,domain=-14.036155147522754:13.957221388439766] plot(\x,{(--16-0*\x)/8});
\draw [line width=2pt, red,domain=-14.036155147522754:13.957221388439766] plot(\x,{(--24-0*\x)/8});
\draw [line width=2pt,red, domain=-14.036155147522754:13.957221388439766] plot(\x,{(--24-0*\x)/8});
\draw [line width=2pt,red ,domain=-14.036155147522754:13.957221388439766] plot(\x,{(--32-0*\x)/8});

\draw [line width=2pt] (-4.82,-1.46)-- (2.64,-1.46);
\draw [->,line width=1pt, orange] (-3,-2) -- (-2,-2) node[below, orange] {$v$};
\draw [->,line width=1pt, orange] (-3,-2) -- (-4,-1) node[above, orange] {$u$};

\begin{scriptsize}
\draw [fill=blue] (-5,-2) circle (2.5pt) node[blue, below right] {$v''_1$};

\draw [fill=blue] (3,-2) circle (2.5pt) node[blue, below left] {$v''_n$};

\draw[color=black] (-13.796895519010254,-1.6342401240156403) node {$\bf{l}'$};

\draw [fill=black] (-6,-5) circle (2.5pt) node[left] {$v_2$};

\draw [fill=black] (5,-5) circle (2.5pt) node[right] {$v_{n-1}$};

\draw[color=black] (-13.741681758584292,-0.5851786759223696) node {$\bf{l}'_0$};
\draw[color=black] (-13.741681758584292,-0.014636484854099574) node[above] {$\bf{l}'_1$};
\draw[color=black] (-13.741681758584292,0.9792112028132096) node[above] {$\bf{l}'_2$};
\draw[color=black] (-13.741681758584292,1.9730588904805186) node[above] {$\bf{l}'_3$};
\draw[color=black] (-13.741681758584292,2.9853111649564816) node[above] {$\bf{l}'_4$};
\draw[color=black] (-13.741681758584292,3.979158852623791) node[above] {$\bf{l}'_5$};

\draw [fill=black] (-4.82,-1.46) circle (2pt);
\draw[color=black] (-4.613006701491978,-1.0452933461387166) node[below left] {$v_1$};

\draw [fill=black] (-4.666666666666667,-1) circle (2pt);
\draw[color=black] (-4.465770007022747,-0.5851786759223696) node[left] {$s_0$};
\draw [fill=black] (2.3333333333333335,-1) circle (2pt);
\draw[color=black] (2.5279729802657194,-0.5851786759223696) node {$t_0$};

\draw [fill=black] (2.64,-1.46) circle (2pt) node[right] {$v_n$};

\draw [fill=black] (-4.333333333333333,0) circle (2pt);
\draw[color=black] (-4.134487444466978,0.4086690117449395) node[left] {$s_1$};
\draw [fill=black] (-4,1) circle (2pt);
\draw[color=black] (-3.8032048819112085,1.4209212862209024) node[left] {$s_2$};
\draw [fill=black] (-3.6666666666666665,2) circle (2pt);
\draw[color=black] (-3.471922319355439,2.414768973888212) node[left] {$s_3$};
\draw [fill=black] (-3.3333333333333335,3) circle (2pt);
\draw[color=black] (-3.122235169991016,3.408616661555521) node[left] {$s_4$};
\draw [fill=black] (-3,4) circle (2pt);
\draw[color=black] (-2.7909526074352464,4.420868936031484) node[left] {$s_5$};
\draw [fill=black] (-1,4) circle (2pt);
\draw[color=black] (-0.8032572321006293,4.420868936031484) node {$t_5$};
\draw [fill=black] (-0.3333333333333333,3) circle (2pt);
\draw[color=black] (-0.12228752018043654,3.408616661555521) node {$t_4$};
\draw [fill=black] (0.3333333333333333,2) circle (2pt);
\draw[color=black] (0.5402776049311024,2.414768973888212) node {$t_3$};
\draw [fill=black] (1,1) circle (2pt);
\draw[color=black] (1.2028427300426414,1.4209212862209024) node {$t_2$};
\draw [fill=black] (1.6666666666666667,0) circle (2pt);
\draw[color=black] (1.8654078551541802,0.4086690117449395) node {$t_1$};
\end{scriptsize}
\end{tikzpicture}
\caption{Illustration to notations in proof of Lemma \ref{lem:step_of_ext}.}    \label{pic:not_lines}
\end{figure}
		 

		The proof is in $m + 1$ steps. 
		On $j$th step we will prove that the sequence $w_0$ is $(r_{j} - l_{j})$-periodic, and that that $w_{j'}$ is $(r_{j - j'} - l_{j -j'})$-periodic for each $j' \leq j$. 
		
		{\underline{Step 0.}} Let $x$ be an integer point from the segment $s_0t_0$ such that $x + vr_0$ also belongs to the segment $s_0t_0$.   
		Since $F_0 \in \cal{P}$, applying \eqref{eq:sum0_z2}, we get that 
		\[
			\sum\limits_{t \in F_0} \mbf{w}(x + t) = 0.
		\]
		It terms of $(u,v)$-representation, this is equivalent to 
		\[
			\sum_{i = 0}^m \sum\limits_{j = l_i}^{r_i-1} \mbf{w}(x + (i, j)_{(u, v)}) = 0.
		\]
		Since $\mbf{w}(y) = 0$ for each $y \in D_0 \cap \bb{Z}^2$, we have
		\[
			\sum_{j = 0}^{r_0 - 1} \mbf{w}(x + vj) = 0.
		\]
		Shifting the coordinates by $v$, with the same arguments we obtain
		$$\sum_{j = 0}^{r_0 - 1} \mbf{w}(x + v + vj) = 0.$$
		After subtracting these equations we get $\mbf{w}(x) = \mbf{w}(x + r_0v)$, i.e., $w_0$ is $r_0$-periodic.
		
		\vspace{1cm}
		
		Suppose we completed all steps before step $k$, i.e., $w_{i}$ is $r_{i'} - l_{i'}$-periodic for each $i = 0, 1 \ldots, k-1$ and $i' = 0, 1, \ldots, k - i - 1$.  
		
		\underline{Step $k$.} Let $x$ be an integer point from the segment $s_jt_j$ such that $x + r_0v$ also belongs to the segment $s_jt_j$. As above, applying applying \eqref{eq:sum0_z2}, we get
		\[
			\sum\limits_{t \in F_0} \mbf{w}(x + t) = 0.
		\]
		Similarly to the proof on step $0$, we get 
		\begin{equation} \label{eq:step1}
			\sum_{i = 0}^{j} \sum_{j=l_i}^{r_i - 1} \mbf{w}(x + iu + jv) = 0;
		\end{equation}
		
		 Shifting our considerations by $v$, we get 
		\begin{equation} \label{eq:step1'}
			\sum_{i = 0}^{j} \sum_{j=l_i}^{r_i - 1} \mbf{w}(x + iu + jv + v) = 0,
		\end{equation}
		provided that $x + v$ and $x + (r_0 + 1)v$ belong to $s_jt_j$.
		 Using \eqref{eq:step1} and \eqref{eq:step1'}, we get
		\[
			\sum_{i = 0}^{j} (\mbf{w}(x + iu + l_iv) - \mbf{w}(x + iu + r_i v)) = 0;
		\]
		
		Then
		\[
			\mbf{w}(x + r_0v) = \mbf{w}(x) + 
			\sum_{i= 1}^{j} (\mbf{w}(x + iu + l_i v) - \mbf{w}(x + iu + r_i v)) 
		\]
		for each $x \in s_it_i \cap \bb{Z}^2$ such that $x + r_0v \in s_it_i \cap \bb{Z}^2$ . Then
		\[
			\mbf{w}(x + 2r_0v) = \mbf{w}(x + r_0v) + 
			\sum_{i= 1}^{j} (\mbf{w}(x + iu + (l_i + r_0) v) - \mbf{w}(x + iu + (r_i + r_0) v)). 
		\] 
		By previous steps we have $\mbf{w}(x + iu + kv) = \mbf{w}(x + iu + (k + r_0)v)$ for each $x$  such that $x + iu + kv$ and $x + iu + (k + r_0)v$ belong to $s_it_i \cap \bb{Z}^2$ 
		This sequence contains at most $|A|$ distinct integer numbers. 
		If $N(|A|, \cal{P})$ is big enough, namely $N(|A|, \cal{P}) > (|A| + 2) r_0 ||v||$,		
		then $\mbf{w}(x) = \mbf{w}(x + r_0v)$, i.e., the sequence $w_i$ is $r_0$-periodic and
		\begin{equation} \label{eq:step3}
			\sum_{i = 1}^{j} (\mbf{w}(x + iu + l_iv) - \mbf{w}(x + iu + r_i v)) = 0.
		\end{equation}
		
		 Now, in the same way, we prove that $w_{i-1}$ is $(r_1 - l_1)$-periodic. Since $w_i$ is $r_0$-periodic, by \eqref{eq:step3} we obtain 
		\[
			\mbf{w}(x + u + (r_1- l_1)v) = \mbf{w}(x + u) +
			\sum_{i= 2}^{j} (\mbf{w}(x + iu + (l_i-l_1) v) - \mbf{w}(x + iu + (r_i -l_1) v)).
		\]
		Using the same arguments, we prove that 
		\[ \mbf{w}(x), \mbf{w}(x + (r_1 - l_1)v), \ldots\] 
		is a long arithmetic progression and hence it is constant. So, $w_{j-1}$ is $(r_1 - l_1)$-periodic. 
		
		With similar arguments we prove that $w_{j-2}$ is $(r_2 - l_2)$-periodic, $w_{j-3}$ is $(r_3 - l_3)$-periodic, \ldots, $w_0$ is $(r_k - l_k)$-periodic.

		After $m + 1$ steps, we get that $w_0$ is $(r_j - l_j)$-periodic for each $j{=0,\ldots,m}$. Then $w_0$ is $\gcd(r_0 - l_0, r_1 -l_1, \ldots, r_n - l_n)$-periodic. Since $\cal{P}$ is relatively prime, we have $\gcd(r_0 - l_0, r_1 -l_1, \ldots r_n - l_n) = 1$. Then $w_0$ is constant and then $w_0$ is an all-zero sequence. 
		So, if we put $N(|A|, \cal{P}) > |\cal{P}| \cdot diam(\cal{P}) \cdot (|A| + 2)$, 
	    where $diam(\cal{P})$ is the diameter of the convex hull of $F_0$, then the progressions are long enough for our goals. 
	\end{proof}
	
	Now we prove Proposition \ref{lem:full_ext}.
	
	\begin{proof}[Proof of Proposition \ref{lem:full_ext}]
		Let $O$ be a point inside the polygon $D_0$ and $1 + \alpha$ be a coefficient of homothety, where $\alpha$ is a small positive real number that we define later.
		Let $D_1$ be the image of $D_0$ by homothety with the center $O$ and the coefficient $1 + \alpha$. Let $v_1, v_2, \ldots, v_n$ be the vertices of $D_0$ and $u_1, u_2, \ldots , u_n$ be the vertices of $D_1$ such that $H(v_i) = u_i$.
		
		The idea of the proof of this proposition is as follows. We will split the polygon $D_1$ into several polygons $G_1, \dots, G_n$ (see Fig.~\ref{pic:ill_prop2}), and, using Lemma \ref{lem:step_of_ext}, we show consecutively that in each of them the word is equal to $0$. Now we proceed with the details of this process.
		
		Let $d_i$ be the distance between neighbouring parallel lines that are parallel to the edge $v_iv_{i+1}$. We set $\delta = \min(d_1, d_2, \ldots d_n)$. 
		
		Let $v_n'$ and $v_1'$ be the intersections of the lines $v_{n-1} v_{n}$ and $v_1v_2$ with the edge $u_n u_1$. If $N(|A|, \cal{P})$ is big enough, then we can apply Lemma \ref{lem:step_of_ext} several times and prove that $\mbf{w}$ is equal to $0$ inside the polygon $G_1$ with vertices  $v_1'v_2v_3 \ldots v_{n-1}v_n'$. If there exist $m_1$ integer lines between $v_1v_n$ and $u_1u_n$, then we apply Lemma~\ref{lem:step_of_ext} $m_1$ times for the lines parallel to the edge $v_1v_n$.
		
		So, we just proved that $\mbf{w}$ is equal to $0$ inside of the polygon $G_1$. Let $v_2'$ be the intersection of the lines $v_2v_3$ and $u_1u_2$. Now, we apply Lemma~\ref{lem:step_of_ext} $m_2$ times. As a result, we prove that $\mbf{w}$ is equal to $0$ inside of the polygon $G_2$ with vertices $v_n'u_1v_2'v_3v_4 \ldots v_{n-1}$. 
		
		Continuing this line of reasoning, we prove that $\mbf{w}$ is  equal to $0$ inside  of the polygon $D_1$. On the $j$'th  step we show that $\mbf{w}$ is equal to $0$ inside of the polygon $G_i$ with vertices $v_n'u_1u_2 \ldots u_{i-1}v_i'v_{i+1} \ldots v_{n-1}$, where $v_i'$ is the intersection of the lines $v_iv_{i+1}$ and $u_{i-1}u_i$ (see Fig.~\ref{pic:ill_prop2}).
		
		So, on the $(n-1)$'th step we show that $\mbf{w}$ is equal to $0$ inside of the polygon $G_{n-1} = u_1u_2 \ldots u_{n-2}v_{n-1}'v_n'$.  
		Finally we apply  Lemma~\ref{lem:step_of_ext} $m_n$ times and show that $\mbf{w}(x) = 0$ for each integer point of the polygon $D_1$.
		
		
		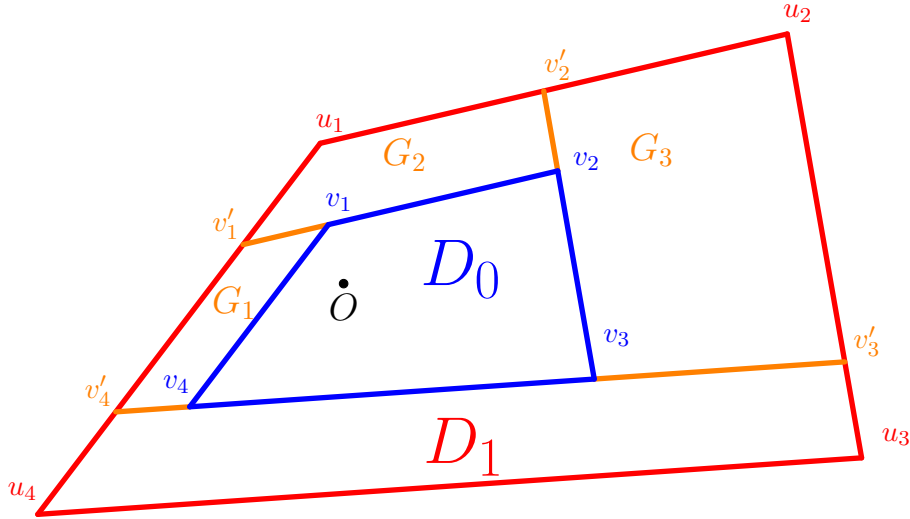
\begin{figure}[h]
		\centering
		\begin{tikzpicture}[line cap=round,line join=round,>=triangle 45,x=1cm,y=1cm, scale = 0.4]
\clip(-17.166005088003452,-10.709169869344567) rectangle (16.846070428679404,12.215338309037909);
\draw [line width=2pt, red] (-15.26657342644998,-7.906019208918661)-- (-5.979524707650827,4.395247076508287);
\draw [line width=2pt, red] (-5.979524707650827,4.395247076508287)-- (9.376691814223216,8.0204546904255);
\draw [line width=2pt, red] (9.376691814223216,8.0204546904255)-- (11.820652003380888,-6.032316397231113);
\draw [line width=2pt, red] (11.820652003380888,-6.032316397231113)-- (-15.26657342644998,-7.906019208918661);

\draw [line width=2pt, orange] (-10.27,-4.34)-- (-12.701318409902713,-4.5081814238428946);
\draw [line width=2pt, orange] (-5.71,1.7)-- (-8.51412051659032,1.0380192944919397);
\draw [line width=2pt, orange] (1.83,3.48)-- (1.369117825711954,6.130072502156264);
\draw [line width=2pt, orange] (3.03,-3.42)-- (11.267241473937329,-2.8502058529306513);

\draw [line width=2pt,color=blue] (-5.71,1.7)-- (1.83,3.48);
\draw [line width=2pt,color=blue] (1.83,3.48)-- (3.03,-3.42);
\draw [line width=2pt,color=blue] (3.03,-3.42)-- (-10.27,-4.34);
\draw [line width=2pt,color=blue] (-10.27,-4.34)-- (-5.71,1.7);

\begin{small}
\draw[blue] (-5.363392549400713,2.436459319727894) node {$v_1$};
\draw[blue] (2.0, 3.8) node[right] {$v_2$};
\draw[blue] (3.0,-2.6748796199114566) node[above right] {$v_3$};
\draw[blue] (-9.923808729078933,-3.5930842198466695) node[left] {$v_4$};

\draw[fill = black] (-5.210358449411511,-0.256940840082063) circle (4pt) node[below] {\large{$O$}};

\draw[red] (-14.943327208724758,-7.143475339596158) node[left] {$u_4$};
\draw[red] (12.143708489363997,-5.337672959723573) node[right] {$u_3$};
\draw[red] (-5.638853929381276,5.068645839542171) node {$u_1$};
\draw[red] (9.725769709534605,8.710857419285182) node {$u_2$};

\draw[orange] (-12.5,-3.807331959831552) node[left] {$v_4'$};
\draw[orange] (-8.3,1.7325024597775642) node[left] {$v_1'$};
\draw[orange] (12,-2.1545636799481693) node {$v_3'$};
\draw[orange] (1.8292101500917815,7.0) node {$v_2'$};

\draw[blue] (-1.35389912968362,0.23276827988338372) node {\Huge{$D_0$}};
\draw[orange] (-8.8,-0.8996840600367119) node {\large{$G_1$}};
\draw[orange] (-3.2,4.0) node {\large{$G_2$}};
\draw[orange] (4.920498969873662,4.181048059604798) node {\large{$G_3$}};
\draw[red] (-1.262078669690099,-5.664322919771084) node {\Huge{$D_1$}};
\end{small}
\end{tikzpicture}
		\caption{Illustration to construction of proof of Proposition~\ref{lem:full_ext}} \label{pic:ill_prop2}
		\end{figure}
		
		Now, we define $N(|A|, \cal{P})$. We can say that all elementary extensions (i.e., the results of applying Lemma \ref{lem:step_of_ext} once) can make the corresponding edge shorter by at most $\Delta$. This number depends on the pattern only. So, the minimal length of the edge in our proof should be at least 
		\[ N - \max(m_1, m_2 \ldots m_n) \cdot \Delta. \]
		Since 
		\[ m_i < \frac {\alpha dist(O, v_iv_{i-1})} {\delta},\] 
		we have 
		\[ \max(m_1, m_2 \ldots , m_n) < \frac {\alpha diam(D_0)} {\delta}.\] 
		We set $\alpha = \frac 1 {diam(D_0)}$. 
				Then we can take
		\[ N(|A|, \cal{P}) = \dfrac {\Delta} {\delta} + diam(\cal{P}) \cdot |\cal{P}| \cdot (|A| + 2).\]
	\end{proof}
	
	Since $diam(D_1) = diam(D_0) + 1$, the following is a direct corollary of Proposition \ref{lem:full_ext}: 
	
	\begin{corollary} \label{cor:sum_then0}
		Let $A$ be a finite subset of  complex numbers containing
		$0$. 
		Let $\cal{P}$ be a relatively prime convex pattern and $\bf{w}$ be a two-dimensional word on $A$ such that 
		\[ \sum_{x \in F} {\bf{w}}(x) = 0\]
		for each figure $F \in \cal{P}$. 
		Let $D$ be the convex hull of the figure $F$.
		If there exists a number $N(|A|, \cal{P})$ 
		and a polygon $D_0$ parallel to $D$ such that
		\begin{itemize}
			\item any edge of $D_0$ contains at least $N$ integer points, where $N > N(|A|, \cal{P})$;
			\item $\mbf{w}(x) = 0$ for each $x \in D_0 \cap \bb{Z}^2$.
		\end{itemize}
		Then $\mbf{w}(x) = 0$ for each $x$ in $\bb{Z}^2$.
	\end{corollary}
	
	
 	
 	Now using this corollary we are ready to prove the following theorem, which is a stonger form of the ``if'' part of Theorem \ref{th:main_convex} stated not only for words, but also for sums of values of integer-valued  functions:
 	
 	\begin{theorem} \label{th:sum_then_per}
		Let $A$ be a finite subset of complex numbers and $\cal{P}$ be a convex relatively prime pattern. Let $\mbf{w}$ be a two-dimensional word on the alphabet $A$ and 
		\[
			\sum\limits_{x \in F} \mbf{w}(x) = C
		\] 
		for each $F \in \cal{P}$ and some complex number $C$. Then $\mbf{w}$ is fully periodic.  
	\end{theorem}

	\begin{proof}
		Let $F_0$ be a figure of the pattern $\cal{P}$ and $D_0$ be the convex hull of the figure $F_0$. Suppose that $D$ is a polygon such that
		\begin{itemize}
			\item $D$ and $D_0$ are parallel;
			\item any edge of $D$ is bigger than $N(|A|^2, \cal{P})$. 
		\end{itemize}
		
		Let $u, v$ be a basis of $\bb{Z}^2$. We let $F_D$ denote the figure $D \cap \bb{Z}^2$. 
		By pigeonhole principle, there exist $i_1$ and $i_2$ such that 
		$\mbf{w}(x + i_1u) = \mbf{w}(x + i_2u)$ for each $x \in F_D$.  
		
		Consider a new two-dimensional word $\mbf{z}$ such that $\mbf{z}(x) = \mbf{w}(x + i_1u) - \mbf{w}(x + i_2u)$. By Corollary \ref{cor:sum_then0}, we have $\mbf{z}(x) = 0$ for each $x \in \bb{Z}^2$. 
		Then $\mbf{w}(x) = \mbf{w}(x + (i_2 - i_1)u)$, i.e., $\mbf{w}$ is $(i_2- i_1)u$-periodic. 
		
		In the same way we can find another periodicity vector $(j_2 - j_1)v$; it follows that $\mbf{w}$ is fully periodic. 
	\end{proof}
	
	Now, we conclude this section by proving the Theorem \ref{th:sum_then_per} and hence our main result, Theorem \ref{th:main_convex}.
	
	\begin{proof}[Proof of Theorem \ref{th:main_convex}.] 
		Let $\mbf{w}$ be a word on an  alphabet $A=\{a_1,\ldots, a_n\}$ and $a_{\bf{w}}(\cal{P}) = 1$. 
		We can associate a complex number with each letter of the alphabet, so that there is a bijection between $A$ and a finite subset $A'$ of complex numbers. For example, we can take $A'=\{1, \ldots, n\}$ and build a bijection $f$ as $f(a_i)= i$. Let $\mbf{w}'$ be the word on the alphabet $A'$ such that $\mbf{w}'(x) = f(\mbf{w}(x))$ for each $x \in \bb{Z}^2$. Evidently, $a_{\mbf{w}'}(\cal{P}) = a_{\mbf{w}}(\cal{P})$. 
		Since $a_{\mbf{w}'}(\cal{P}) = 1$, we have 
		\[ \sum\limits_{x \in F} \mbf{w}(x) = C. \]
		for each  $F\in \cal{P}$ and some integer $C$.
		By Theorem~\ref{th:sum_then_per}, $\mbf{w}'$ is fully periodic. So, $\mbf{w}$ is also fully periodic.
	\end{proof}

	\section{Conclusions and open problems}
	
A related problem although in different terminology has been considered in  \cite{AP08}. In that paper, periodicity of centered functions on infinite regular grids has been studied. Let $G=(V,E)$ be a graph, $r$ be an integer. A function $f:V\to\mathbb{R}$ is called {\emph{centered}} of radius $r$ if the sum of its values in each ball of radius $r$ is equal to some constant. One can consider $\mathbb{Z}^2$ as a graph, where the vertices are pairs of integers, and the edges connect vertices which differ in one coordinate by $1$. The main result of \cite{AP08} is the following:

	\begin{theorem}  \cite{AP08}  Let $f: \mathbb{Z}^2 \to \mathbb{Z}$ be a bounded centered function of radius $r\geq1$. Then $f$ is periodic. \end{theorem}
	
	In fact, this theorem can be obtained as a direct corollary of Theorem \ref{th:main_convex}, where the pattern corresponds to a ball of radius $r$. The idea of the method used in \cite{AP08}, the method of so-called $R$-prolongable words, is somewhat similar to the method used in the current paper. 
	Both methods rely on certain areas with the same values of the function which have to be extended in the same way. 
	In fact, this paper gives a broad generalization of the results from \cite{AP08} from a very restricted shapes (balls of a given radius) to arbitrary convex figures. 
	The proofs in the general case are much more technical, and the statement of the general result is given in an algebraic form in terms of polynomials and their divisors. The proof is combinatorial; it would be interesting to try to find another proof based on some algebraic or symbolic dynamical  techniques.

	

	
	\bigskip
	
	In the paper we consider words for which abelian pattern complexity equals $1$ for some pattern $\cal{P}$  in the one- and two-dimensional cases. In the one-dimensional case we get that all words of this class are periodic. Besides that, we characterized patterns $\cal{P}$ for which there exists a function from $\mathbb{Z}$ to $\mathbb{C}$ with finitely many values such that it has constant sums over figures from $\cal{P}$. This gives a necessary condition on a pattern for existence of non-trivial words with abelian pattern complexity equal to one for the pattern.
	
	In the two-dimensional case we get a characterization of convex patterns  $\cal{P}$ for which  $a_{\mbf{w}}(\cal{P}) = 1$ only for fully periodic words $\mbf{w}$ (see Theorem~\ref{th:main_convex}). 
	
	An interesting problem is finding a generalization  of Theorem~\ref{th:main_convex} for non-convex patterns, i.e., characterizing patterns $\cal{P}$ such that abelian pattern complexity for these pattern equals 1 only for periodic words.  We suspect that our algebraic characterization from Theorem~\ref{th:main_convex} might also hold also for non-convex patterns, as well as for patterns with weights. Another interesting problem is to find a characterization of abelian rigid patterns in dimension higher than $2$.

\section{Acknowledgements}

	This research is supported by the Foundation for the Advancement of Theoretical Physics and Mathematics
``BASIS''.

\end{document}